\documentclass[12pt]{amsart}

\usepackage{lipsum} 
\usepackage[utf8]{inputenc} 
\usepackage[T1,OT2,T1]{fontenc} 

\usepackage[dvipsnames]{xcolor} 
\usepackage{mathtools} 
\usepackage{etoolbox} 
\usepackage[shortlabels]{enumitem} 
\usepackage{hhline, diagbox} 
\calclayout
\let\oldbibliography\thebibliography
\renewcommand{\thebibliography}[1]{\oldbibliography{#1}
\setlength{\itemsep}{5pt}}

\usepackage[hypertexnames=false]{hyperref}
\hypersetup{colorlinks = true,
  linkcolor = Blue,
  urlcolor  = Blue,
citecolor = OliveGreen}
\usepackage[noabbrev,nameinlink]{cleveref} 

\usepackage{amsfonts, amscd, amsmath, amssymb, amsbooka, amsthm, mathrsfs}
\usepackage{mathdots, stmaryrd} 

\usepackage{tikz-cd}
\usetikzlibrary{arrows.meta, positioning, quotes, decorations.markings, decorations.pathreplacing, calligraphy, shapes, matrix, fit, positioning, backgrounds}
\tikzcdset{scale/.style={every label/.append
    style={scale=#1},
cells={nodes={scale=#1}}}}
\tikzset{->-/.style={decoration={
      markings,
mark=at position 0.5 with {\arrow{>}}},postaction={decorate}}}
\tikzset{->>-/.style={decoration={
      markings,
mark=at position 0.5 with {\arrow{>>}}},postaction={decorate}}}
\tikzset{->>>-/.style={decoration={
      markings,
mark=at position 0.5 with {\arrow{>>>}}},postaction={decorate}}}
\tikzcdset{vertical/.style={
    anchor=south,
    rotate=90,
inner sep=.5mm}}

\newtheorem{thm}[subsection]{Theorem}
\newtheorem{prop}[subsection]{Proposition}

\newtheorem{cor}[subsection]{Corollary}

\theoremstyle{remark}
\newtheorem*{rem}{Remark}
\newtheorem*{ex}{Example}
\newtheorem*{exs}{Examples}
\newtheorem*{asd}{Aside}
\theoremstyle{definition}
\newtheorem{df}[subsection]{Definition}

\makeatletter
\tagsleft@false\let\veqno\@@eqno 
\makeatother

\usepackage{graphicx}

\title[A Morse Mayer-Vietoris Sequence for Cosheaf Homology]{A Morse Mayer-Vietoris Sequence for Cosheaf Homology on Simplicial Complexes}
\author{Ben Gould}
\address{Mathematical Institute, University of Oxford, Andrew Wiles Building, Radcliffe Observatory Quarter, Woodstock Road, Oxford, OX2 6GG, UK}
\email{bengould2003@gmail.com}

\begin{document}

\begin{abstract}
  We investigate the homology of cosheaves over finite simplicial complexes. After constructing the Mayer-Vietoris short exact sequence for this homology theory, we apply discrete Morse theory to this setting, defining the associated Morse chain complex. The main result is the construction of a \emph{Morse Mayer-Vietoris} short exact sequence of Morse chain complexes, for which we establish a quasi-isomorphism to the standard Mayer-Vietoris sequence. Discussions relating poset-based (co)sheaf definitions to topological ones via Kan extensions and exploring a modern categorical perspective on discrete Morse theory are also included.
\end{abstract}

\maketitle
\tableofcontents

\section{Introduction}

The primary objective of this essay, as outlined in the project guidance, is to study cosheaf homology on simplicial complexes, with an aim to explore how we can use \emph{discrete Morse theory} (as found in Vidit Nanda's lecture notes \cite{Nan25}) to simplify the computation of the \emph{Mayer-Vietoris sequence}. Following the suggested structure, we will define cosheaves on simplicial complexes and their homology theory (\Cref{section:1}), construct the standard Mayer-Vietoris sequence for these cosheaves (\Cref{section:3}), introduce the Morse chain complex  (\Cref{section:4}), and combine these elements to form the \emph{Morse Mayer-Vietoris sequence} (\Cref{section:6}).

However, this essay shifts some focus from purely computational aspects towards establishing connections with related mathematical areas, aiming to provide deeper theoretical context. Most notably, \Cref{section:2} relates our definitions of sheaves and cosheaves to more standard definitions by following parts of Justin Curry's thesis \cite{Cur14}, and \Cref{section:5} outlines a new point of view developed by Vidit Nanda in \cite{Nan19}, which better situates discrete Morse theory within a more modern homotopy theoretic perspective.

Candidly speaking, the choice of the aforementioned extra topics was made to align with the author's recent interest in category theory. Unfortunately, the inclusion of these discussions has resulted in a rather liberal interpretation of the guideline:
\begin{quote}
  ``It is not \underline{expected} that any submission should exceed 15 pages.''
\end{quote}

\begin{rem}
  Throughout this essay, we use $K$ to refer to a (locally oriented) finite simplicial complex. We use the notation $K_{k}$ for the set of simplices of $K$ with dimension $k$. For an oriented simplex $\sigma=(v_{0},\dots,v_{k})$, we write $\sigma_{-i}$ for the $i$'th face given by $$\sigma_{-i}:=(v_{0},\dots,v_{i-1},v_{i+1},\dots,v_{k}).$$This is the same notation used in the lecture notes \cite{Nan25}.
\end{rem}

\section{Cosheaf homology}\label{section:1}
In this section, we define cosheaves and introduce their basic homology theory, as well as including some examples along the way. Recall from \cite{Nan25} that we define a sheaf on a simplicial complex as a functor $$\mathscr{S}:(K,\leq)\overset{}{\longrightarrow}\mathsf{Vect}_{\mathbb{F}}^{\mathrm{fd}}$$where we have considered $K$ as a poset category of simplices ordered by the face relation. We can define a cosheaf dually:

\begin{df}

  Let $K$ be a finite simplicial complex. A \emph{cosheaf} over $K$ (taking values in $\mathsf{Vect}_{\mathbb{F}}^{\mathrm{fd}}$) is a contravariant functor $$\mathscr{C}:(K,\leq)^{\mathrm{op}}\to \mathsf{Vect}_{\mathbb{F}}^{\mathrm{fd}}.$$

  We write $\mathscr{C}_{\sigma}:=\mathscr{C}(\sigma)$ and call this the \emph{costalk} of $\mathscr{C}$ at $\sigma$. Given an inclusion $\sigma'\leq \sigma$, we will write $\mathscr{C}_{\sigma\geq\sigma'}$ for the map $\mathscr{C}(\sigma\geq \sigma'):\mathscr{C}_{\sigma}\to \mathscr{C}_{\sigma'}$ and call this the \emph{extension map} associated to the inclusion $\sigma\geq\sigma'$. (Note that the order in the subscript $\sigma\geq\sigma'$ is reversed relative to the face relation $\sigma'\leq\sigma$ to clarify the map's direction.)

  We define a \emph{morphism of cosheaves} simply as a natural transformation between the functors. Hence, it will be beneficial to introduce the notation:
  \begin{align*}\mathsf{Shv}(K)&:=\mathsf{Fun}((K,\leq),\mathsf{Vect}^{\mathrm{fd}}_{\mathbb{F}}), \\ \mathsf{CoShv}(K)&:=\mathsf{Fun}((K,\leq)^{\mathrm{op}},\mathsf{Vect}^{\mathrm{fd}}_{\mathbb{F}}).
  \end{align*}

  These form Abelian categories because functor categories inherit the Abelian property from the target category, and $\mathsf{Vect}_{\mathbb{F}}^{\mathrm{fd}}$ is Abelian.
\end{df}

\begin{rem}
  Those familiar with (co)sheaves in other contexts may wish to reconcile these definitions with more standard definitions. We address this in \Cref{section:2}.
\end{rem}

\clearpage

\begin{exs}
  \
  \begin{itemize}
    \item There is a \emph{zero} cosheaf, which we denote $\underline{0}_{K}$, which assigns the zero vector space to every simplex. This also happens to be the zero object in the category of cosheaves on $K$.
    \item The \emph{skyscraper} cosheaf at a simplex $\tau \in K$, $\underline{\mathrm{Sk}}_{\tau}$, which assigns the zero vector space to any simplex $\sigma\neq \tau$ and $\mathbb{F}$ to the simplex $\tau$.
    \item The \emph{constant} cosheaf, $\underline{\mathbb{F}}_{K}$, which takes value $\mathbb{F}$ for all costalks and the identity map for all extension maps.
    \item We can take \emph{direct sums of cosheaves}, which is induced by the direct sum in $\mathsf{Vect}^{\mathrm{fd}}_{\mathbb{F}}$ pointwise. Explicitly, if $\mathscr{C}_{1}$ and $\mathscr{C}_{2}$ are cosheaves on $K$, the direct sum is given by
      \begin{align*}(\mathscr{C}_{1}\oplus \mathscr{C}_{2})({\sigma})&:=\mathscr{C}_{1}(\sigma)\oplus \mathscr{C}_{2}(\sigma),\\(\mathscr{C}_{1}\oplus \mathscr{C}_{2})(\sigma\geq \tau)&:=\mathscr{C}_{1}(\sigma\geq \tau)\oplus \mathscr{C}_{2}(\sigma\geq \tau).
      \end{align*}
    \item Given a cosheaf $\mathscr{C}$ on a subcomplex $L\subseteq K$, we extend $\mathscr{C}$ to a cosheaf on $K$ by
      \begin{align*}\widehat{\mathscr{\hat{}C}}(\sigma)&:=
        \begin{cases}\mathscr{C}_{\sigma}&\text{for $\sigma \in L$},\\0&\text{otherwise},
        \end{cases}\\ \widehat{\mathscr{\hat{}C}}(\sigma\geq \tau)&:=
        \begin{cases}\mathscr{C}_{\sigma\geq \tau}&\text{for $\sigma,\tau \in L$},\\0&\text{otherwise}.
        \end{cases}
      \end{align*}This is called the \emph{extension by zero} cosheaf for $\mathscr{C}$ on $L\subseteq K$.
  \end{itemize}
\end{exs}

We will now define the chain complex associated to a cosheaf $\mathscr{C}$ on a simplicial complex $K$ dually to how the cochain complex for a sheaf is defined in \cite{Nan25}. For ease of notation, it is standard to define the following \emph{incidence symbol} for any two $\sigma,\tau \in K$ as $$[\sigma:\tau]:=
\begin{cases} +1 &\text{if $\tau=\sigma_{-i}$ for even $i$,} \\ -1 &\text{if $\tau=\sigma_{-i}$ for odd $i$,} \\ 0 &\text{otherwise.}
\end{cases}$$

\begin{df}\label{definition:1.2.}

  Given a simplicial complex $K$ and a cosheaf $\mathscr{C}:(K,\leq)^{\mathrm{op}}\to \mathsf{Vect}_{\mathbb{F}}^{\mathrm{fd}}$ on $K$, we define the \emph{chain complex} on $K$ with $\mathscr{C}$-coefficients in each dimension as the vector space, $$C_{k}(K;\mathscr{C}):=\bigoplus_{\sigma \in K_{k}}\mathscr{C}(\sigma)\,,$$ where the direct sum notation is justified as our simplicial complexes are finite (and $\mathsf{Vect}_{\mathbb{F}}^{\mathrm{fd}}$ is an additive category). When $K_{k}$ is empty for some $k\in \mathbb{N}$, the definition above gives the empty direct sum which is the zero vector space.

  We define the boundary (linear) maps $\partial_{k}^{\mathscr{C}}:C_{k}(K;\mathscr{C})\to C_{k-1}(K;\mathscr{C})$ on each component $\mathscr{C}({\sigma})$ (a vector space) of the direct sum $C_{k}(K;\mathscr{C})$ by the map $$\partial_{k}^{\mathscr{C}}\big|_{\sigma}:=\sum_{\tau\in K_{k-1}} [\sigma:\tau]\cdot \mathscr{C}(\sigma\geq \tau)$$Note that we are abusing notation slightly by writing $\mathscr{C}({\sigma\geq \tau})$ where this may not be defined, but in these cases the incidence symbol is zero anyway. We now prove that this does indeed form a chain complex.
\end{df}

To give a more intuitive description of the boundary maps, let's define the symbol $$\partial_{k}^{\mathscr{C}}\big|_{\sigma\geq \tau}:=[\sigma:\tau]\cdot \mathscr{C}(\sigma\geq \tau).$$for the block of the matrix for $\partial^{\mathscr{C}}_{k}$ corresponding to the map $\mathscr{C}_{\sigma}\to \mathscr{C}_{\tau}$. This is shown in the following picture:
\[
  \partial^{\mathscr{C}}_{k}:=
  \begin{tikzpicture}[baseline=(M.center),]
    \matrix (M) [matrix of math nodes,
      left delimiter={[}, right delimiter={]},
      row sep=0em, column sep=1em,
      nodes={anchor=center}
    ]
    {
      {\raisebox{1mm}{$\ddots$}} & {\raisebox{1mm}{$\vdots$}} & {\raisebox{1mm}{$\iddots$}} \\
      \cdots & {\partial^{\mathscr{C}}_{k}\big|_{\sigma\geq\tau}} & \cdots \\
      {\raisebox{1mm}{$\iddots$}} & {\raisebox{1mm}{$\vdots$}} & {\raisebox{1mm}{$\ddots$}} \\
    };
    \node[above=0.5em of M-1-1] {$\cdots$};
    \node[above=0.5em of M-1-2] {$\mathscr{C}(\sigma)$};
    \node[above=0.5em of M-1-3] {$\cdots$};
    \node[left=1.2em of M-1-1]  {\raisebox{-3mm}{$\vdots$}};
    \node[left=0.5em of M-2-1]  {$\mathscr{C}(\tau)$};
    \node[left=1.2em of M-3-1]  {\raisebox{2.5mm}{$\vdots$}};
    \begin{scope}[on background layer]
      \node [fit=(M-1-2)(M-2-2)(M-3-2),
        fill=gray!25,
      inner sep=0pt] {};
      \node [fit={([shift={(-9.3mm,0mm)}]M-2-1.north east)(M-2-2)([shift={(2.4mm,0mm)}]M-2-3.north east)},
        fill=gray!25,
      inner sep=0pt] {};
      \node [fit=(M-2-2), fill=gray!50, inner sep=0pt] {};
      \node [fit=(M-2-2), draw, inner sep=0mm, line width=0.5mm] {};
    \end{scope}
  \end{tikzpicture}
\]
\begin{ex}

  The chain complex associated to the constant cosheaf $\underline{\mathbb{F}}_{K}$ recovers the usual chain complex associated to the simplicial complex $K$, $C_{\bullet}(K;\mathbb{F})$, as discussed in \cite{Nan25}.
\end{ex}

\begin{prop}

  If we have a cosheaf $\mathscr{C}:(K,\leq)^{\mathrm{op}}\to \mathsf{Vect}_{\mathbb{F}}^{\mathrm{fd}}$. Then $C_{k}(K;\mathscr{C})$, as defined above, does indeed form a chain complex.
\end{prop}
\begin{proof}{[With the way we have set up our notation, the proof becomes very similar to the proof in \cite{Nan25} for the sheaf cochain complex.]}

  We have defined an $\mathbb{N}$-indexed sequence of finite dimensional vector spaces with linear maps between them as required for a chain complex over the Abelian category of finite dimensional vector spaces. So we must check that the boundary maps satisfy $\partial^{2}=0$.

  To do this, we just have to check the identity on a single block of the composite matrix, say corresponding to $\mathscr{C}_{\tau}\to \mathscr{C}_{\tau''}$, for some $K_{k+1}\ni\tau\geq\tau''\in K_{k-1}$:

  \begin{align*}
    \partial^{\mathscr{C}}_{k}\circ \partial_{k+1}^{\mathscr{C}}
    &=\sum_{\tau' \in K_{k}} \partial^{\mathscr{C}}_{k-1}\big|_{\tau'\geq \tau''}\circ \partial^{\mathscr{C}}_{k}\big|_{\tau\geq \tau'} \\
    &=\sum_{\tau\geq \tau'\geq \tau''}[\tau:\tau']\cdot[\tau':\tau'']\cdot \mathscr{C}_{\tau'\geq \tau''}\circ \mathscr{C}_{\tau\geq \tau'} \\
    &=\sum_{\tau\geq \tau'\geq \tau''}[\tau:\tau']\cdot[\tau':\tau'']\cdot\mathscr{C}_{\tau\geq \tau''}\tag{$\ast$}\label{eq:1}\\
    &=\left[ \sum_{\tau\geq \tau'\geq \tau''} [\tau:\tau']\cdot[\tau:\tau''] \right]\cdot\mathscr{C}_{\tau\geq \tau''}
  \end{align*} where in the second equality, we re-index as to ignore all zero terms of the sum.

  Observe that the scalar term in the large brackets is independent of the specific cosheaf $\mathscr{C}$. Choosing the constant cosheaf $\underline{\mathbb{F}}_{K}$ recovers the standard simplicial chain complex $C_{\bullet}(K;\mathbb{F})$, for which the boundary map squares to zero: $$\left[ \sum_{\tau\geq \tau'\geq \tau''} [\tau:\tau']\cdot[\tau:\tau''] \right]\cdot \mathrm{id}_{\mathbb{F}}=\partial^{K}_{k}\circ \partial^{K}_{k+1}=0.$$Since this holds universally, the composition $\partial^{\mathscr{C}}_{k}\circ \partial_{k+1}^{\mathscr{C}}=0$ for any cosheaf $\mathscr{C}$ confirming that $(C_{\bullet}(K;\mathscr{C}),\partial^{\mathscr{C}}_{\bullet})$ does indeed form a chain complex.
\end{proof}

\begin{df}

  Given a cosheaf $\mathscr{C}:(K,\leq)^{\mathrm{op}}\to \mathsf{Vect}^{\mathrm{fd}}_{\mathbb{F}}$ on $K$, we define the \emph{cosheaf homology groups} $$H_{k}(K;{\mathscr{C}}):=\frac{\text{Ker }\partial^{\mathscr{C}}_{k}}{\text{Im } \partial^{\mathscr{C}}_{k+1}}$$This is well-defined as $\partial^{\mathscr{C}}_{k}$ is just a map of vector spaces.
\end{df}

\begin{rem}
  The chain complexes defined above naturally live in the (Abelian) chain complex category of finite dimensional vector spaces, denoted $\mathsf{Ch}_{\bullet}(\mathsf{Vect}^{\mathrm{fd}}_{\mathbb{F}})$, which has morphisms called \emph{chain maps} $\varphi_{\bullet}:C_{\bullet}\to D_{\bullet}$. These chain maps are defined as a collection of maps making the following diagram commute:

  \[
    \begin{tikzcd}[scale=1, column sep=9mm, row sep=7mm]
      \cdots \ar[r, "\partial_{k+1}"]  & C_{k} \ar[d, "\varphi_{k}"]  \ar[r, "\partial_{k}"] & C_{k-1} \ar[d, "\varphi_{k-1}"] \ar[r, "\partial_{k-1}"] & \cdots \\
      \cdots \ar[r, "\partial_{k+1}"]  & D_{k} \ar[r, "\partial_{k}"] & D_{k-1} \ar[r, "\partial_{k-1}"] & \cdots
    \end{tikzcd}
  \]

  With this perspective, the homology defined above is a special case of the more general homology functor $H_{k}:\mathsf{Ch}_{\bullet}(\mathcal{A})\to \mathcal{A}$ defined for chain complexes on an arbitrary Abelian category $\mathcal{A}$. This functor is defined on objects in the same way as in the above definition and on morphisms by
  \begin{align*}
    H_{k}\left( C_{\bullet}\overset{\varphi_{\bullet}}{\longrightarrow} D_{\bullet} \right)&:H_{k}C_{\bullet}\to H_{k}D_{\bullet}, \\
    [z]&\mapsto[\varphi_{k}z].
  \end{align*}(This is shown to be well-defined by showing that chain maps send boundaries to boundaries and cycles to cycles.)
\end{rem}

\begin{ex}
  In the field of \emph{topological data analysis}, it can be relevant to study the homology of fibres of simplicial maps $F:K\to L$. This data is coherently (retaining the homology data for the fibre inclusion maps) organised into a sheaf. A detailed motivation for this and definition of the sheaf is found in \cite{Nan25}. Dually, we use a cosheaf to organise the relevant data for fibre cohomology.

  Explicitly, given a simplicial map $f:K\to L$, we first define, for $\tau \in L$, $$\tau/f:=\{\sigma \in K \mid f(\sigma)\leq \tau\}.$$Then the $k^{\text{th}}$ \emph{fibre cohomology cosheaf of} $f$ is defined by
  \begin{align*}\mathscr{C}_{\tau}&:=H^{k}(\tau/f),  \\ \mathscr{C}_{\tau\geq\tau'}&:=H^{k}(\tau/f)\to H^{k}(\tau'/f),
  \end{align*}where $H^{k}(-):\mathsf{SimpCx}\to \mathsf{Vect}$ denotes the standard simplicial homology functor (with coefficients in $\mathbb{F}$) and the extension map is the map induced by the inclusion $\tau'/f\hookrightarrow\tau/f$.
\end{ex}

\section{Different notions of (co)sheaves}\label{section:2}

Readers familiar with the notion of (co)sheaves in other contexts may feel uncomfortable with the definitions we gave in the previous section. This section aims to reconcile these definitions with the more familiar definition of a (co)sheaf on a topological space.

\begin{df}
  A \emph{sheaf} (resp. \emph{cosheaf}) \emph{on a topological space} $X$ valued in a category $\mathcal{D}$ is defined as a covariant (contravariant) functor from the open set category of $X$ to the category $\mathcal{D}$, subject to a certain \emph{gluing condition}:

  Let $X$ be a topological space. We say a functor $F:\mathsf{Open}(X)^{\mathrm{op}}\to\mathcal{D}$ satisfies the \emph{sheaf condition} if for any open set $U\in \mathsf{Open}(X)$ and open cover $\{U_{j} \}_{j\in J}$\footnote{Assumed to be non-repeating, i.e. $U_{j}=U_{k} \iff j=k$.} of $U$, the following diagram is an \emph{equaliser diagram}: $$F(U)\to\prod_{j}^{}F(U_{j})\rightrightarrows\prod_{x,y\in J}^{}F(U_{x}\cap U_{y}),$$where the first map is induced by the inclusion maps $U_{j}\hookrightarrow U$ with the universal property of the product, and the parallel maps are induced by the product of the relevant inclusions maps.

  Note that this condition only makes sense if $\mathcal{D}$ actually has the products involved in the statement of the condition. If we are in a position to state the sheaf condition, we can define the \emph{category of sheaves on $X$ taking values in $\mathcal{D}$}, $\mathsf{Shv}(X;\mathcal{D})$, as the full subcategory of objects of $\mathsf{Fun}(\mathsf{Open}(X)^{\mathrm{op}},\mathcal{D})$ satisfying the sheaf condition. We define the cosheaf condition and category of cosheaves dually (as a \emph{coequaliser diagram}).
\end{df}

So our definition of a (co)sheaf on a simplicial complex has three glaring problems:

\begin{enumerate}
  \item Our functors are not from a category of open sets of a topological space.
  \item Our sheaf definition isn't contravariant (and our cosheaf definition isn't covariant).
  \item There is no gluing condition.
\end{enumerate}

Surprisingly, all three problems can be addressed by taking the right perspective. We follow the work found in chapters two and four of \cite{Cur14} (which streamlined the theory developed in \cite{Lad08}), using adapted definitions and merging a few relevant results into one. This culminates in \Cref{thm:2.5}, which describes a general equivalence between (co)sheaves defined on posets and cosheaves defined on an associated topological space. We will first introduce some definitions.

\begin{df}
  Given a set $P$, we say that a category $\mathcal{D}$ is $|P|$\emph{-}(\emph{co})\emph{complete} if either $P$ is finite and $\mathcal{D}$ is finitely (co)complete or $P$ is infinite and $\mathcal{D}$ is (co)complete\footnote{This notation is introduced so that we can amend the statements of some results in \cite{Cur14} so that they apply to the case where we take a cosheaf on a (finite) simplicial complex with values in the (finitely (co)complete) category $\mathsf{Vect}^{\mathrm{fd}}_{\mathbb{F}}$.}.
\end{df}

\begin{df}
  A \emph{sheaf} (resp. \emph{cosheaf}) \emph{on a partially ordered set} $(P,\leq)$ taking values in a $|P|$-(co)complete category $\mathcal{D}$ is defined as a functor $(P,\leq)\to \mathcal{D}$ or $(P,\leq)^{\mathrm{op}}\to \mathcal{D}$ resp. We write $\mathsf{Shv}(P;\mathcal{D})$ and $\mathsf{CoShv}(P;\mathcal{D})$ for the respective functor categories.
\end{df}

\begin{df}
  Any poset $(P,\leq)$ can be equipped with a topology called the \emph{Alexandrov topology}, whose open sets are given by the subsets $U$ of $P$ satisfying the condition $$\text{if $x\leq y$ and $x\in U$, then $y\in U$.}$$We write $P_{\mathrm{Alex}}$ to refer to the poset $P$ as a topological space with the above topology.

  At any $p\in P$, define the \emph{open star} at $p$ as the set $\mathcal{U}_{p}:=\{x\in P \mid p\leq x\}$ and note the collections of open stars form a basis for $P_{\mathrm{Alex}}$. In the case when $P$ is a simplicial complex, the open star at a simplex coincides with the definition of the star of a simplex given in \cite{Nan25}.

  We also have a contravariant \emph{inclusion functor}, $$i:(P,\leq)\overset{}{\longrightarrow}\mathsf{Open}(P_{\mathrm{Alex}})^{\mathrm{op}},$$defined on objects by sending $p\in P$ to $\mathcal{U}_{p}$ and sends order relations between elements to inclusions of open stars.
\end{df}

This inclusion functor hints at how we can resolve problems $(1)$ and $(2)$ above; we will \emph{extend} our functor $(P,\leq)\to \mathcal{D}$ (resp. $(P,\leq)^{\mathrm{op}}\to \mathcal{D}$) to a functor from $\mathsf{Open}(P_{\mathrm{Alex}})^{\mathrm{op}}$ to $\mathcal{D}$ (resp. $\mathsf{Open}(P_{\mathrm{Alex}})$ to $\mathcal{D}$) such that
\[
  \begin{tikzcd}[scale=1, column sep=6mm, row sep=7mm]
    (P,\leq) \ar[r, ""] \ar[d, "i"']  & \mathcal{D} \\
    \mathsf{Open}(P_{\mathrm{Alex}})^{\mathrm{op}} \ar[ru, ""',dashed]
  \end{tikzcd}
\]
(Similarly for a functor $\mathsf{Open}(P_{\mathrm{Alex}})\to \mathcal{D}$, but with $i^{\mathrm{op}}$ instead of $i$)

We do this using the theory of \emph{Kan extensions} - which refers to a certain way of `universally extending' a functor (e.g. a (co)sheaf of a poset) along another functor (e.g. $i$ or $i^{\mathrm{op}}$).

There are actually two dual notions for the Kan extension of a functor $F:\mathcal{C}\to \mathcal{E}$ along $K:\mathcal{C}\to \mathcal{D}$. These are called the `left' and `right' Kan extensions, and are denoted $\mathrm{Lan}_{K}F$ and $\mathrm{Ran}_{K}F$ respectively. For readers wanting a detailed reference on Kan extensions, the author finds the treatment in chapter 6 of \cite{Rie16} adequate. The following result is needed to reach our goal.

\begin{thm}
  Given functors $F:\mathcal{C}\to \mathcal{E}$ and $K:\mathcal{C}\to \mathcal{D}$, where $\mathcal{C}$ is small and $\mathcal{D}$ is $|\mathrm{ob}\,\mathcal{C}|$-\emph{(}co\emph{)}complete, there is a formula for the left and right Kan extensions given on objects by
  \begin{align*}
    \mathrm{Lan_{K}F(d)=\mathrm{colim}}\Big((K\!\downarrow\!d)\overset{\Pi^{d}}{\longrightarrow}\mathcal{C}\overset{F}{\longrightarrow}\mathcal{E}\Big), && \mathrm{Ran_{K}F(d)=\mathrm{lim}}\Big((d\!\downarrow\!K)\overset{\Pi_{d}}{\longrightarrow}\mathcal{C}\overset{F}{\longrightarrow}\mathcal{E}\Big)
  \end{align*}where $(K\!\downarrow\!d)$ and $(d\!\downarrow\!K)$ denote the comma categories $(K\!\downarrow\!\mathrm{const}_{d})$ and $(\mathrm{const}_{d}\!\downarrow\!K)$ respectively, and the $\Pi$ functors are the canonical projections.
\end{thm}

\begin{rem}
  These formulae determine the action of the Kan extensions on morphisms by the universal property of colimits.
\end{rem}

Now we are ready to state the main result that relates the definition of (co)sheaves on posets to (co)sheaves on topological spaces and justifies why we needed to define the Alexandrov topology.

\begin{thm}\label{thm:2.5}
  Let $(P,\leq)$ be a poset. Let $\mathcal{D}$ be $|P|$-\emph{(}co\emph{)}complete. Then there are equivalences of categories
  \begin{align*}
    \mathsf{Shv}(P;\mathcal{D})\cong\mathsf{Shv}(P_{\mathrm{Alex}};\mathcal{D}),  && \mathsf{CoShv}(P;\mathcal{D})\cong\mathsf{CoShv}(P_{\mathrm{Alex}};\mathcal{D}),
  \end{align*}given by taking right and left Kan extensions respectively.
\end{thm}
\begin{proof}(\emph{some details omitted for brevity})
  We will focus on the equivalence of categories on the right as this is most relevant in the rest of this essay. The left equivalence then follows by a dual argument.\\[12pt]
  Given any cosheaf $\mathscr{C}:(P,\leq)^{\mathrm{op}}\to \mathcal{D}$ on $P$, the left Kan extension, $$G:=\mathrm{Lan}_{i}\mathscr{C}:\mathsf{Open}(P_{\mathrm{Alex}})\to \mathcal{D,}$$(where we are now writing $i$ for $i^{\mathrm{op}}:(P,\leq)^{\mathrm{op}}\to\mathsf{Open}(P_{\mathrm{Alex}})$) gives us a functor of the right form to be a cosheaf. We now must check the cosheaf condition.

  To do this, we first fix an arbitrary open set $U\in P_{\mathrm{Alex}}$, and consider the open cover $\{\mathcal{U}_{p}\}_{p\in P_{\mathrm{Alex}}}=:\mathcal{\underline{U}}$. The cosheaf condition is satisfied if the following diagram is a coequaliser:
  \[
    \begin{tikzcd}[scale=1, column sep=9mm, row sep=7mm]
      \coprod\limits_{x,y\in U}G(\mathcal{U}_{x}\cap \mathcal{U}_{y}) \ar[r, "", shift left] \ar[r, ""', shift right]  & \coprod\limits_{p\in U}G(\mathcal{U}_{p}) \ar[r, ""] & G(U).
    \end{tikzcd}
  \]

  The formula for the left Kan extension gives $$G(U):={\text{colim }}\Big( (i\!\downarrow\!U)\overset{\Pi^{{U}}}{\longrightarrow}(P,\leq)^{\mathrm{op}}\overset{\mathscr{C}}{\longrightarrow}\mathcal{\mathcal{D}}\Big).$$The definition of the comma category here simplifies greatly; it is equivalent to a category with objects given as all the $p\in U$, and a unique morphism $p\to q$ whenever $p\geq q$ in $P$. Under this equivalence, we see that $G(U)$ is given by the colimit over the diagram in $\mathcal{D}$ consisting of objects $\mathscr{C}_{p}$ for each $p\in U$ and a unique morphism $\mathscr{C}_{p\geq q}:\mathscr{C}_{p}\to \mathscr{C}_{q}$ whenever $p\geq q$. Unwinding the definitions for the coproducts allows us to see that this diagram is in fact a coequaliser.

  Next we consider an arbitrary open cover, $\{U_{j}\}_{j\in J}=:\underline{U}$, of the open set $U$, and aim to show

  \[
    \begin{tikzcd}[scale=1, column sep=9mm, row sep=7mm]
      \coprod\limits_{x,y\in J}G(U_{x}\cap U_{y}) \ar[r, "", shift left] \ar[r, ""', shift right]  & \coprod\limits_{j\in J}G(U_{j}) \ar[r, ""] & G(U),
    \end{tikzcd}
  \]
  is always a coequaliser diagram. The crucial observation here is that $\mathcal{\underline{U}}$ always refines\footnote{We say $\{V_{\beta}\}_{\beta\in B}$ \emph{refines} $\{W_{\alpha}\}_{\alpha\in A}$ if for any $\beta\in B$ there exists $\alpha\in A$ such that $V_{\beta}\subseteq W_{\alpha}$} another open cover of $U$ as $\mathcal{U}_{p}$ always includes into the $U_{j}$ containing the point $p$. This refinement property can be used to construct morphisms:
  \[
    \begin{tikzcd}[scale=1, column sep=9mm, row sep=2.5mm]
      \coprod\limits_{x,y\in J}G(U_{x}\cap U_{y}) \ar[r, "", shift left] \ar[r, ""', shift right]  & \coprod\limits_{j\in J}G(U_{j}) \ar[rd, "",start anchor={[xshift=-0.5ex,yshift=1ex]}]
      \\
      && G(U).
      \\
      \coprod\limits_{x,y\in U}G(\mathcal{U}_{x}\cap \mathcal{U}_{y}) \ar[uu, "",dashed,yshift=1ex]  \ar[r, "", shift left] \ar[r, ""', shift right]  & \coprod\limits_{p\in U}G(\mathcal{U}_{p}) \ar[ru, "",end anchor={[xshift=0ex,yshift=0.25ex]}, start anchor={[xshift=-0.5ex]}] \ar[uu, "",dashed,yshift=1ex]
    \end{tikzcd}
  \]
  Using the fact that the lower diagram is a coequaliser as well as some universal properties, a diagram chase reveals that $G(U)$ with its canonical morphism is a cocone of the top coequaliser diagram. Hence the universal property of the coequaliser for the cosheaf condition for $\underline{U}$ gives
  \[
    \begin{tikzcd}[scale=1, column sep=9mm, row sep=6mm]
      \coprod\limits_{x,y\in J}G(U_{x}\cap U_{y}) \ar[r, "", shift left] \ar[r, ""', shift right]  & \coprod\limits_{j\in J}G(U_{j}) \ar[rdd, "",start anchor={[xshift=-1.5ex,yshift=1.5ex]}] \ar[r, ""]  & \mathrm{Coeq}(\underline{U}) \ar[dd, "\exists!", dashed]
      \\
      \\
      \coprod\limits_{x,y\in U}G(\mathcal{U}_{x}\cap \mathcal{U}_{y}) \ar[uu, "",yshift=1ex]  \ar[r, "", shift left] \ar[r, ""', shift right]  & \coprod\limits_{p\in U}G(\mathcal{U}_{p}) \ar[r, ""] \ar[uu, "",yshift=1ex] & G(U)
    \end{tikzcd}
  \]
  Finally, commutativity of the diagram, gives that $\mathrm{Coeq}(\underline{U})$ is a cocone over the bottom coequaliser diagram and hence there is a unique $G(U)\to \mathrm{Coeq}(\underline{U})$ implying the previous unique morphism was an isomorphism and so $\underline{U}$ satisfies the cosheaf condition. We have now shown that $G:=\mathrm{Lan}_{i}\mathscr{C}$ satisfies the cosheaf condition for all open sets $U$ with arbitrary open cover $\{U_{j}\}_{j\in J}$ so taking left Kan extensions does indeed give a functor $\mathsf{CoShv}(P;\mathcal{D})\to \mathsf{CoShv}(P_{\mathrm{Alex}};\mathcal{D})$.

  For the functor in the other direction; given a cosheaf $G:\mathsf{Open}(P_{\mathrm{Alex}})\to \mathcal{D}$, we define a functor $\mathscr{C}:(P,\leq)^{\mathrm{op}}\to \mathcal{D}$ by the following cofiltered limit: $$\mathscr{C}(p):=\varprojlim_{U\ni p}\;G(U)= G(\mathcal{U}_{p}),$$where the equality at the end is due to $\mathcal{U}_{p}$ being the smallest neighbourhood containing $p$. This functor in the other direction which forms an equivalence of the categories.
\end{proof}

\begin{cor}\label{corollary:2.6}
  The previous theorem applies to our definition of (co)sheaves on a simplicial complex. So we have an equivalence of categories: $$\mathsf{Shv}(K):=\mathsf{Fun}((K,\leq),\mathsf{Vect}^{\mathrm{fd}}_{\mathbb{F}})\cong\mathsf{Shv}(P_{\mathrm{Alex}};\mathsf{Vect}^{\mathrm{fd}}_{\mathbb{F}}),$$
\end{cor}
\begin{proof}
  The subtlety here is that the category of \emph{finite} dimensional vector spaces (unlike $\mathsf{Vect}_{\mathbb{F}}$) is not actually complete and cocomplete, but rather \emph{finitely} complete and \emph{finitely} cocomplete. However, as our simplicial complexes are taken to be finite, the limits and colimits involved in the statement of the (co)sheaf condition are always going to be indexed over finite diagrams. Hence, \Cref{thm:2.5} applies.
\end{proof}

\vspace{\baselineskip}

\section{The Mayer-Vietoris sequence}\label{section:3}

In this section, we define how we can restrict cosheaves, allowing us to introduce the Mayer-Vietoris short exact sequence.

\begin{df}
  Given a subcomplex inclusion $L\hookrightarrow K$, we form a \emph{restriction functor} $$(-)|_{L}:\mathsf{CoShv}(K)\to \mathsf{CoShv}(L).$$It is given on objects $\mathscr{C}$ by the composition $$(L,\leq)^{\mathrm{op}}\overset{\iota}{\longrightarrow}(K,\leq)^{\mathrm{op}}\overset{\mathscr{C}}{\longrightarrow}\mathsf{Vect}^{\mathrm{fd}}_{\mathbb{F}}.$$On morphisms $f:\mathscr{C}_{1}\to\mathscr{C}_{2}$, restriction is given by the following horizontal composition of natural transformations:

  \[
    \begin{tikzcd}[column sep=huge]
      (L,\leq)^{\mathrm{op}}
      \ar[bend left=25]{r}[name=A,label=above:$\iota$, yshift=-1.5mm]{}
      \ar[bend right=25]{r}[name=B,label=below:$\iota$]{} &
      (K,\leq)^{\mathrm{op}}
      \ar[bend left=25]{r}[name=C,label=above:$\mathscr{C}_{1}$, yshift=-1.5mm]{}
      \ar[bend right=25]{r}[name=D,label=below:$\mathscr{C}_{2}$]{} &
      \mathsf{Vect}^{\mathrm{fd}}_{\mathbb{F}}
      \ar[shorten <=7pt,shorten >=3pt,Rightarrow,to path={(A) -- node[label=right:$\mathrm{id}_{}$, yshift=-0.75mm, xshift=-0.6mm] {} (B)}]{}
      \ar[shorten <=7pt,shorten >=3pt,Rightarrow,to path={(C) -- node[label=right:$f$, yshift=-0.7mm, xshift=-0.65mm] {} (D)}]{}
    \end{tikzcd}
  \]
\end{df}

\vspace{\baselineskip}

Now we are ready to state and prove the cosheaf homology version of the familiar \emph{Mayer-Vietoris sequence} from algebraic topology.

\begin{thm}\label{theorem:3.2} Let $L,M$ be subcomplexes of $K$ such that $K=L\cup M$ and set $I:=L\cap M$. If $\mathscr{C}$ is a cosheaf on $K$, then there is a short exact sequence of chain complexes: $$0\overset{}{\longrightarrow}C_{\bullet}(I;\mathscr{C}|_{I})\overset{p_{\bullet}}{\longrightarrow}C_{\bullet}(L;\mathscr{C}|_{L})\oplus C_{\bullet}(M;\mathscr{C}|_{M})\overset{q_{\bullet}}{\longrightarrow}C_{\bullet}(K;\mathscr{C})\overset{}{\longrightarrow}0.$$The middle term is given by the direct sum of chain complexes in the abelian category of chain complexes $\mathsf{Ch}_{\bullet}(\mathsf{Vect}^{\mathrm{fd}}_{\mathbb{F}})$.
\end{thm}
\begin{proof}
  We first claim that the inclusions $I\hookrightarrow L$ and $I\hookrightarrow M$ induce inclusions of chain complexes $$i_{\bullet}:C_{\bullet}(I;\mathscr{C}|_{I})\hookrightarrow C_{\bullet}(L;\mathscr{C}|_{L})\;\text{ and }\;j_{\bullet}:C_{\bullet}(I;\mathscr{C}|_{I})\hookrightarrow C_{\bullet}(M;\mathscr{C}|_{M}),$$
  which are defined levelwise via the following decomposition (shown for the left inclusion),
  \begin{align*}
    C_{k}(L;\mathscr{C}|_{L})&=\bigoplus_{\sigma \in L_{k}}\mathscr{C}|_{L}(\sigma)\\
    &=\underbrace{\Big[\bigoplus_{\sigma \in I_{k}}\mathscr{C}|_{L}(\sigma)\Big] }_{C_{k}(I;\mathscr{C}|_{I})}\oplus\Big[\bigoplus_{\sigma \in L_{k}\setminus I_{k}}\mathscr{C}|_{L}(\sigma)\Big]
  \end{align*}

  Next, we check the naturality condition that the inclusion maps $i_{k}$ and $i_{k-1}$ commute with the boundary operators. To see this, note that the restriction of the boundary operator of $C(L;\mathscr{C}|_{L})$ to the copy of $C_{k}(I;\mathscr{C}|_{I})$ in the decomposition above is given on each $\sigma \in I$ by $$\partial_{k}^{\mathscr{C}|_{L}}\big|_{\sigma}:=\sum_{\sigma\geq \tau \in L} [\sigma:\tau]\cdot \mathscr{C}|_{L}(\sigma\geq \tau)=\sum_{\sigma\geq \tau \in I} [\sigma:\tau]\cdot \mathscr{C}|_{I}(\sigma\geq \tau)=\partial_{k}^{\mathscr{C}|_{I}}\big|_{\sigma}$$where the second equality follows because $I$ is a simplicial complex and so contains all of its faces, and $\mathscr{C}|_{L}(\sigma\geq \tau)=\mathscr{C}|_{I}(\sigma\geq \tau)$ for $\sigma,\tau \in I$ by definition.
  The second inclusion $j_{\bullet}$ follows by an identical argument.

  The direct sum of chain complexes is given in the obvious way by: $$(C_{\bullet},\partial_{\bullet}^{C})\oplus(D_{\bullet},\partial_{\bullet}^{D})=(C_{\bullet}\oplus D_{\bullet},\partial_{\bullet}^{C}\oplus\partial_{\bullet}^{D}).$$

  Now we can define $p_{\bullet}=i_{\bullet}\oplus j_{\bullet}$, giving an injective chain map as required.

  To define the map $q_{\bullet}$, we first note that the inclusions $L\hookrightarrow K$ and $M\hookrightarrow K$ induce inclusions at the chain complex level $$i'_{\bullet}:C_{\bullet}(L;\mathscr{C}|_{L})\hookrightarrow C_{\bullet}(K;\mathscr{C})\;\text{ and }\;j'_{\bullet}:C_{\bullet}(M;\mathscr{C}|_{M})\hookrightarrow C_{\bullet}(K;\mathscr{C}),$$in the same way as before. Then, because $K=L\cup M$ implies every summand $\mathscr{C}_{\sigma}$ of $C(K;\mathscr{C})$ is also a summand of $C_{\bullet}(L;\mathscr{C}|_{L})$ or $C_{\bullet}(M;\mathscr{C}|_{M})$, we deduce that $$q_{\bullet}:=i'_{\bullet}\oplus -j'_{\bullet}:C_{\bullet}(L;\mathscr{C}|_{L})\oplus C_{\bullet}(M;\mathscr{C}|_{M})\overset{}{\longrightarrow} C_{\bullet}(K;\mathscr{C})$$is a surjective chain map.

  We now show exactness at the middle term. We present a matrix argument as the author believes it gives good intuition for the $q_{\bullet}$ map. Take $C_{k}(K;\mathscr{C})$ and $C_{\bullet}(L;\mathscr{C}|_{L})\oplus C_{\bullet}(M;\mathscr{C}|_{M})$ to have ordered bases with respect to decomposition $$K_{k}=(L_{k}\setminus M_{k})\cup(M_{k}\cap L_{k})\cup(M_{k}\setminus L_{k}).$$Then the matrices for $i'_{k}$ and $-j'_{k}$ with respect to this ordered basis for $K$ are given by

  \[
    i'_{k}=
    \begin{tikzpicture}[baseline=(m.center), node distance = 7mm,]
      \matrix (m) [matrix of nodes, nodes in empty cells,
        nodes={anchor=center, inner sep=4pt},
        row 1/.style={nodes={minimum height=1cm}},
        row 2/.style={nodes={minimum height=0.5cm}},
        column 1/.style={nodes={minimum width=1cm}},
        left delimiter={[},
        right delimiter={]}
      ]
      {
        \scalebox{1.2}{$\mathrm{id}$} \\
        \scalebox{1.2}0 \\
      };
      \draw[black,very thick] ([xshift=-1mm,yshift=0.5mm]m-1-1.south west) -- ([xshift=1mm,yshift=0.5mm]m-1-1.south east);

      \draw[decorate,
        decoration={calligraphic brace, amplitude=1mm,
          pre =moveto, pre  length=0.5mm,
          post=moveto, post length=0mm,
        raise=5mm},
        very thick,
      pen colour=black]   (m-2-1.north east) -- node[right=6mm] {$M_{k}\backslash L_{k}$}  (m-2-1.south east);
      \draw[decorate,
        decoration={calligraphic brace, amplitude=1.5mm,
          pre =moveto, pre  length=0mm,
          post=moveto, post length=1mm,
        raise=5mm},
        very thick,
      pen colour=black]   (m-1-1.north east) -- node[right=6mm] {$L_{k}$}  (m-1-1.south east);
    \end{tikzpicture}
    \text{and}
    \quad
    -j'_{k}=
    \begin{tikzpicture}[baseline=(m.center)]
      \matrix (m) [matrix of nodes, nodes in empty cells,
        nodes={anchor=center, inner sep=4pt},
        row 2/.style={nodes={minimum height=1cm}},
        row 1/.style={nodes={minimum height=0.5cm}},
        column 1/.style={nodes={minimum width=1cm}},
        left delimiter={[},
        right delimiter={]}
      ]
      {
        \scalebox{1.2}0 \\
        \scalebox{1.2}{$-\mathrm{id}$} \\
      };
      \draw[black,very thick] ([xshift=-1mm,yshift=-0.5mm]m-1-1.south west) -- ([xshift=1mm,yshift=-0.5mm]m-1-1.south east);

      \draw[decorate,
        decoration={calligraphic brace, amplitude=1.5mm,
          pre =moveto, pre  length=1mm,
          post=moveto, post length=0mm,
      raise=4.5mm}, very thick]
      (m-2-1.north east) -- node[right=6mm] {$M_{k}$}  (m-2-1.south east);
      \draw[decorate,
        decoration={calligraphic brace, amplitude=1mm,
          pre =moveto, pre  length=0mm,
          post=moveto, post length=0.5mm,
      raise=5mm}, very thick]
      (m-1-1.north east) -- node[right=6mm] {$L_{k}\backslash M_{k}$}  (m-1-1.south east);
    \end{tikzpicture}
  \]
  Hence $q_{k}$ is given by the block matrix:
  \[
    q_{k}=
    \begin{tikzpicture}[baseline=(m.center)]
      \matrix (m) [matrix of nodes, nodes in empty cells,
        nodes={anchor=center, inner sep=0pt},
        row 1/.style={nodes={minimum height=1cm}},
        row 2/.style={nodes={minimum height=0.5cm}},
        row 3/.style={nodes={minimum height=1cm}},
        column 1/.style={nodes={minimum width=1cm}},
        column 2/.style={nodes={minimum width=0.5cm}},
        column 3/.style={nodes={minimum width=1cm}},
        left delimiter={[},
        right delimiter={]}
      ]
      {
        \scalebox{1.2}{$\mathrm{id}$} && \scalebox{1.2}{0}\\
        & \scalebox{1.2}{0} & \\
        \scalebox{1.2}{0} && \scalebox{1.2}{$\mathrm{-id}$} \\
      };
      \draw[black,very thick] ([xshift=-2mm]m-2-1.south west) -- (m-2-2.south east);
      \draw[black,very thick] ([yshift=1mm]m-1-2.north east) -- (m-2-2.south east);
      \draw[black,very thick] (m-2-2.north west) -- ([xshift=2mm]m-2-3.north east);
      \draw[black,very thick] (m-2-2.north west) -- ([yshift=-1mm]m-3-2.south west);
      \draw[decorate,
        decoration={calligraphic brace, amplitude=1mm,
          pre =moveto, pre  length=0mm,
          post=moveto, post length=1mm,
      raise=5mm}, very thick]
      (m-1-3.north east) -- node[right=6mm,yshift=1mm] {$L_{k}\backslash M_{k}$}  (m-1-3.south east);
      \draw[decorate,
        decoration={calligraphic brace, amplitude=0.75mm,
          pre =moveto, pre  length=0mm,
          post=moveto, post length=0mm,
      raise=5mm}, very thick]
      (m-2-3.north east) -- node[right=6mm] {$L_{k}\cap M_{k}$}  (m-2-3.south east);
      \draw[decorate,
        decoration={calligraphic brace, amplitude=1mm,
          pre =moveto, pre  length=1mm,
          post=moveto, post length=0mm,
      raise=5mm}, very thick]
      (m-3-3.north east) -- node[right=6mm,yshift=-1mm] {$M_{k}\backslash L_{k}$}  (m-3-3.south east);
    \end{tikzpicture}
  \]
  This allows us to see that $\text{Ker }q_{\bullet}=\text{Im }p_{\bullet}$ and so the sequence is exact.
\end{proof}

\begin{cor}
  We have a long exact sequence of Homology groups (vector spaces in this case):
  $$\cdots\to H_{k}(I;\mathscr{C}|_{I})\to H_{k}(L;\mathscr{C}|_{L})\oplus H_{k}(M;\mathscr{C}|_{M})\to H_{k}(K;\mathscr{C})\to H_{k-1}(I;\mathscr{C}|_{I})\to\cdots$$
\end{cor}

\section{The Morse complex}\label{section:4}

Our goal in this section is to apply \emph{discrete Morse theory} to cosheaf homology. In particular we will introduce the \emph{Morse chain complex} and describe how it can simplify the computation of homology. First let us recall some definitions from \cite{Nan25}.

Recall that we write $\sigma\lhd\tau$ when $\sigma$ is a \emph{facet} of $\tau$, meaning $\sigma$ is a codimension one face of $\tau$; this is called the covering relation. A \emph{partial matching} on $K$ is a collection $\Sigma=\{(\sigma_{\bullet}\lhd\tau_{\bullet})\}$ of simplex-pairs such that if $(\sigma\lhd\tau)\in \Sigma$, then neither $\sigma$ nor $\tau$ appear in any other $\Sigma$-pair. (Note this definition can be given more generally for any finite poset). We call a simplex $\Sigma$-\emph{critical} if it is not contained in the partial matching data.

A $\Sigma$-\emph{path} in $K$ is a \emph{zig-zag} of simplices $$\rho=(\sigma_{1}\lhd\tau_{1}\rhd\sigma_{2}\lhd\tau_{2}\rhd\cdots\lhd\tau_{m-1}\rhd\sigma_{m}\lhd\tau_{m}),$$ such that $(\sigma_{i}\lhd\tau_{i})\in\Sigma$ for all $i$. We call a $\Sigma$-path \emph{gradient} if either $m=1$ or $\sigma_{1}$ is not a face of $\tau_{m}$. We say a partial matching $\Sigma$ is \emph{acyclic} if all of its paths are gradient.

\begin{df}
  Given a cosheaf $\mathscr{C}$ on $K$, we will say an acyclic partial matching $\Sigma$ is $\mathscr{C}$\emph{-compatible}
  if the extension maps $\mathscr{C}_{\tau\rhd\sigma}$ are vector space isomorphisms for all $(\sigma\lhd\tau)\in\Sigma$.

  Now fix a $\mathscr{C}$-compatible partial matching $\Sigma$. For any $\Sigma$-path $\rho$ as in the notation above, we define the $\mathscr{C}$-\emph{weight} $\omega_{\rho}^{\mathscr{C}}$ of $\rho$ as the linear map $$\omega_{\rho}^{\mathscr{C}}:=\pm\Big[\mathscr{C}^{-1}_{\tau_{m}\rhd\sigma_{m}}\circ \mathscr{C}_{\tau_{m-1}\rhd\sigma_{m}}\circ\cdots\circ \mathscr{C}^{-1}_{\tau_{2}\rhd\sigma_{2}}\circ\mathscr{C}_{\tau_{1}\rhd\sigma_{2}}\circ  \mathscr{C}^{-1}_{\tau_{1}\rhd\sigma_{1}}\Big],$$ where the $\pm$ symbol is determined by the parity of the formula: $$(-1)^{m}\cdot[\tau_{m}:\sigma_{m}]\cdot[\tau_{m-1}:\sigma_{m}]\cdots[\tau_{2}:\sigma_{2}]\cdot[\tau_{1}:\sigma_{2}]\cdot[\tau_{1}:\sigma_{1}].$$
\end{df}

\begin{thm}\label{thm:4.2}
  In the setup of the above definition, we can form the \emph{Morse chain complex}:
  \[
    \begin{tikzcd}[scale=1, column sep=9mm, row sep=7mm]
      \cdots \ar[r, "\partial^{\mathscr{C},\Sigma}_{k+2}"] & C^{\Sigma}_{k+1}(K;\mathscr{C}) \ar[r, "\partial^{\mathscr{C},\Sigma}_{k+1}"] & C^{\Sigma}_{k}(K;\mathscr{C}) \ar[r, "\partial^{\mathscr{C},\Sigma}_{k}"] & C^{\Sigma}_{k-1}(K;\mathscr{C}) \ar[r, "\partial^{\mathscr{C},\Sigma}_{k-1}"] & \cdots
    \end{tikzcd}
  \]
  which is quasi-isomorphic to $C_{\bullet}(K;\mathscr{C})$, and given in each dimension by $$C_{k}^{\Sigma}(K;\mathscr{C}):=\bigoplus_{\alpha}\mathscr{C}_{\alpha}\,,$$where the direct sum is taken over all $\Sigma$-critical simplices. The boundary maps $\partial_k^{\mathscr{C},\Sigma}$ are given by matrices whose block corresponding to the map $\mathscr{C}_{\alpha}\to\mathscr{C}_{\omega}$ \emph{(}for $\Sigma$-critical simplices $\alpha \in K_{k}$, $\omega \in K_{k-1}$\emph{)} is: $$[\alpha:\omega]_{\Sigma}:=[\alpha:\omega]\cdot\mathscr{C}_{\alpha\rhd\omega}+\sum_{\substack{\Sigma-\text{paths }\rho\\ (\sigma\lhd\cdots\lhd\tau)}}[\tau:\omega]\cdot\mathscr{C}_{\tau\rhd\omega}\circ \omega_{\rho}^{\mathscr{C}}\circ \mathscr{C}_{\alpha\rhd\sigma}\cdot[\alpha:\sigma].$$ We provide a picture for this block matrix similar to the one given in \Cref{definition:1.2.}:
  \[
    \partial^{\mathscr{C},\Sigma}_{k}=
    \begin{tikzpicture}[baseline=(M.center),]
      \matrix (M) [matrix of math nodes,
        left delimiter={[}, right delimiter={]},
        row sep=0em, column sep=1em,
        nodes={anchor=center}
      ]
      {
        {\raisebox{1mm}{$\ddots$}} & {\raisebox{1mm}{$\vdots$}} & {\raisebox{1mm}{$\iddots$}} \\
        \cdots & {[\alpha:\omega]_{\Sigma}} & \cdots \\
        {\raisebox{1mm}{$\iddots$}} & {\raisebox{1mm}{$\vdots$}} & {\raisebox{1mm}{$\ddots$}} \\
      };
      \node[above=0.5em of M-1-1] {$\cdots$};
      \node[above=0.5em of M-1-2] {$\mathscr{C}(\alpha)$};
      \node[above=0.5em of M-1-3] {$\cdots$};
      \node[left=1.2em of M-1-1]  {\raisebox{-3mm}{$\vdots$}};
      \node[left=0.5em of M-2-1]  {$\mathscr{C}(\omega)$};
      \node[left=1.2em of M-3-1]  {\raisebox{2.5mm}{$\vdots$}};

      \begin{scope}[on background layer]
        \node [fit=(M-1-2)(M-2-2)(M-3-2),
          fill=gray!25,
        inner sep=0pt] {};
        \node [fit={([shift={(-9.3mm,0mm)}]M-2-1.north east)(M-2-2)([shift={(2.4mm,0mm)}]M-2-3.north east)},
          fill=gray!25,
        inner sep=0pt] {};
        \node [fit=(M-2-2), fill=gray!50, inner sep=0pt] {};
        \node [fit=(M-2-2), draw, inner sep=0mm, line width=0.5mm] {};
      \end{scope}
    \end{tikzpicture}
  \]
\end{thm}

\begin{proof}[Proof Outline]
  The fact that we have formed a genuine chain complex follows in the same way as the proof of proposition 8.8 in \cite{Nan25}, except the contributions of the relevant sum have some additional linear map terms. Explicitly, for $\alpha,\omega$ critical simplices and $\Sigma=\{\sigma\lhd\tau\}$, the $\alpha,\omega$ block of the double boundary operator's matrix (`$B$') is given by the linear map
  \[
    B=\sum_{\xi\in K\backslash\{\sigma,\tau\}}[\xi:\omega]_{\Sigma}\circ[\alpha:\xi]_{\Sigma},
  \]
  which at each summand $\xi$ is equal to:
  \begin{equation*}
    \begin{split}
      B_{\xi}=&\left([\xi:\omega]\cdot\mathscr{C}_{\xi\rhd\omega} - [\tau:\omega]\cdot[\tau:\sigma]\cdot[\xi:\sigma]\cdot\mathscr{C}_{\tau\rhd\omega}\circ\mathscr{C}_{\tau\rhd\sigma}^{-1}\circ\mathscr{C}_{\xi\rhd\sigma}\right)\\
      &\quad\circ\left([\alpha:\xi]\cdot\mathscr{C}_{\alpha\rhd\xi} - [\tau:\xi]\cdot[\tau:\sigma]\cdot[\alpha:\sigma]\cdot\mathscr{C}_{\tau\rhd\xi}\circ\mathscr{C}_{\tau\rhd\sigma}^{-1}\circ\mathscr{C}_{\alpha\rhd\sigma}\right)\\
      =&\left([\xi:\omega]\cdot\mathscr{C}_{\xi\rhd\omega} - [\tau:\omega]\cdot[\tau:\sigma]\cdot[\xi:\sigma]\cdot\mathscr{C}_{\xi\geq\omega}\right)\\
      &\quad\circ\left([\alpha:\xi]\cdot\mathscr{C}_{\alpha\rhd\xi} - [\tau:\xi]\cdot[\tau:\sigma]\cdot[\alpha:\sigma]\cdot\mathscr{C}_{\alpha\geq\xi}\right)
    \end{split}
  \end{equation*}
  The proof continues in the same way (realising by considering dimensions that only three of the four terms after distributing over the composition remain), leading to:
  \begin{equation*}
    \begin{split}
      B = &\sum_{\xi}[\xi:\omega] \cdot [\alpha:\xi]\cdot {\overbrace{\mathscr{C}_{\xi\rhd\omega}\circ\mathscr{C}_{\alpha\rhd\xi}}^{\mathscr{C}_{\alpha\rhd\omega}}}\\
      &\quad - [\tau:\omega] \cdot  [\tau:\sigma] \cdot \sum_{\xi} [\xi:\sigma] \cdot [\alpha:\xi] \cdot \smash{\overbrace{\mathscr{C}_{\xi\geq\omega} \circ \mathscr{C}_{\alpha\geq\xi}}^{\mathscr{C}_{\alpha\rhd\omega}}}\\
      &\qquad- [\tau:\sigma] \cdot [\alpha:\sigma] \cdot \sum_{\xi} [\xi:\omega] \cdot [\tau:\xi] \cdot \smash{\overbrace{\mathscr{C}_{\xi\geq\omega} \circ
      \mathscr{C}_{\alpha\geq\xi}}^{\mathscr{C}_{\alpha\rhd\omega}}}\\
      =& - \mathscr{C}_{\alpha\rhd\omega}\cdot[\tau:\omega] \cdot [\alpha:\tau] - \mathscr{C}_{\alpha\rhd\omega} \cdot[\sigma:\omega] \cdot [\alpha:\sigma]\\[3px]
      &\quad + \mathscr{C}_{\alpha\rhd\omega} \cdot [\tau:\omega] \cdot [\alpha:\tau] \\[3px]
      &\qquad + \mathscr{C}_{\alpha\rhd\omega}\cdot [\sigma:\omega]\cdot[\alpha:\sigma]\\
      =&\,0
    \end{split}
  \end{equation*}
  where the second equality comes from the \Cref{eq:1} for the cosheaf complex boundary operator.

  The proof for the chain homotopy equivalence can be adapted from Proposition 8.10 of \cite{Nan25} in a similar manner, which implies the quasi-isomorphism.
\end{proof}

While the boundary maps may seem complicated, there are a couple useful things to keep in mind:
\begin{itemize}
  \item If we take our field to be $\mathbb{F}_{2}$, we can ignore all the incidence symbols, which removes much of the computation.
  \item An often useful, visual, and intuitive way to work out paths in the index of the sum is to draw a \emph{Hasse diagram} but with arrows between matched pairs inverted. This can often allow one to quite easily read off the zigzags between any two simplices. Unfortunately, this quickly becomes inefficient once the complex gets too large, but the author believes that drawing a Hasse diagram in small cases can be helpful for building intuition. See the example diagram below, taken from \cite{For02}:
\end{itemize}
\begin{center}
  \includegraphics[width=0.5\linewidth]{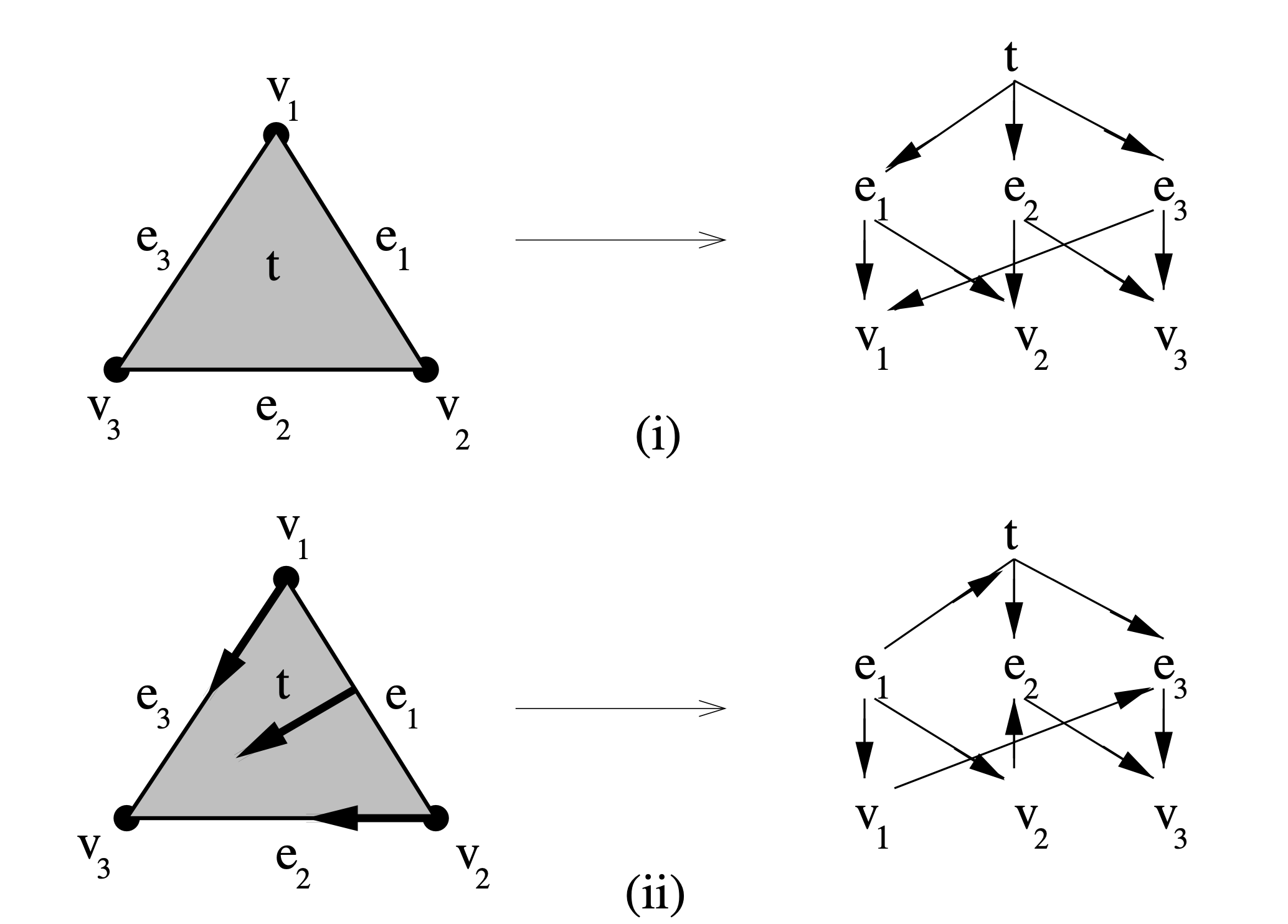}
\end{center}

\begin{rem}\phantomsection\label{rmk:scythe}
  Despite the points above, the Morse chain complex in any real application can not reasonably be computed manually. However, the Morse chain complex offers significant advantages for the \emph{algorithmic} computation of cosheaf homology. Computing cosheaf homology via the Morse complex requires three steps:
  \begin{itemize}
    \item Find a compatible partial matching.
    \item Compute the Morse boundary maps.
    \item Calculate the homology of the resulting complex using standard linear algebra.
  \end{itemize}
  While the final homology calculation needs no specific algorithm, the first two steps do. An algorithm performing these steps for cellular sheaves, aptly named `\texttt{Scythe}', was developed in \cite{CGN16}. This algorithm adapts to work for cellular cosheaves, which include the cosheaves on simplicial complexes we have focused on. The paper's main theorem  demonstrates that `\texttt{Scythe}' reduces the time complexity for these two steps from $O(n^{3}d^{3})$ to $O(np\tilde{m}d^{\omega})$ (here $n$ is the cell count, $d$ bounds the stalk dimension ($\text{sup}_{\sigma}\{\dim\mathscr{S}(\sigma)\}\leq d$), $\tilde{m}=\sum_{k\geq0}{\#\{K_{k}\}^{2}}$, $2\le\omega\le3$ is a constant depending on what matrix multiplication algorithm we use, and $p$ bounds the number of codimension one neighbours for any cell).
\end{rem}

\begin{asd}
  \cite{Nan25} relays a result for sheaf cohomology that is dual to \Cref{thm:4.2}. While this may not surprise the reader, it is not something which should be taken for granted; the general theory of sheaf cohomology for sheaves on topological spaces \emph{does not} formally dualise to obtain a dual theory of cosheaf homology. The reason for this lies in the fact that $\mathsf{CoShv}(X;\mathcal{A})$ does not admit a few important properties that $\mathsf{Shv}(X;\mathcal{A})$ does\footnote{If this seems surprising, it may be helpful to recall that categories are rarely equivalent to their opposite, and a functor category is equivalent to the corresponding contravariant functor category if and only if the source category is equivalent to it's opposite.}. To give some more detail; sheaf cohomology more generally is usually computed as the right derived functors of the (left exact) global sections functor $\Gamma$: $$H^{k}(X;\mathcal{F}):=R^{k}(\mathcal{F}):=H^{k}(\Gamma(X,\mathcal{I}^{\bullet})),$$ where $\mathcal{I}^{\bullet}$ is some acyclic resolution of $\mathcal{F}$. Attempting to dualise this definition by taking left derived functors of the cosections functor confronts us with the problem that $\mathsf{CoShv}(X;\mathcal{A})$ generally lacks enough projectives to construct the necessary resolutions. This technical obstruction may explain why historically, the development of cosheaf theory has been hindered. More recently, however, it has been shown (in \cite{Pra16}) that under mild hypotheses ($\mathcal{A}$ being locally presentable), $\mathsf{CoShv}(X;\mathcal{A})$ \emph{does} admit enough projectives. A less abstract proof applicable to the specific context of cosheaves on posets can be found in chapter 11 of \cite{Cur14}.
\end{asd}

\section{Categorification of discrete Morse theory}\label{section:5}

We will now discuss a more abstract perspective on discrete Morse theory (DMT), presented in \cite{Nan19}. This work provides a categorical framework for viewing DMT, taking inspiration from an idea in ordinary Morse theory (described in \cite{CJS70} - an unpublished manuscript of Cohen, Jones, and Segal). This perspective not only aligns with our goals for this essay, but also can be used to understand and apply DMT in more abstract contexts.

We first note that a \emph{derived} perspective is perhaps foreshadowed by the presence of \emph{zig-zags} of simplices and the definition of $\mathscr{C}$-compatibility. Readers who share this intuition will not be surprised to learn that Nanda takes localisations of categories as a main focal point.

\cite{Nan19} works at the generality of \emph{cellular categories}, which are poset enriched categories (more naturally thought of as a strict 2-category) such that each Hom-poset has a weakly indecomposable minimal element. We call such an element (morphism) \emph{atomic}. These cellular categories are a generalisation of the \emph{entrance path category} associated to any finite poset, which we define now.

\begin{df}\label{definition:5.1}
  Given a finite poset $P$ (usually a CW complex), the associated \emph{entrance path category} $\mathcal{E}(P)$ is defined by
  \begin{align*}
    \mathrm{Ob}(\mathcal{E}(P))&=P_{0}, \\
    \mathrm{Hom}_{\mathcal{E}(P)}(x,y)&=\mathrm{Path}_{P}(x,y)
  \end{align*}where $\mathrm{Path}_{P}$ denotes the poset of \emph{strictly decreasing} paths from $x$ to $y$ in $P$. The partial order enrichment is induced by the obvious inclusion of paths and composition given by concatenation of paths. This has atomic morphisms given by the shortest path, $x\geq y$.

  We can recover the original poset by collapsing all the Hom posets onto their atoms, via the canonical \emph{projection} functor $$\Pi:\mathcal{E}(P)\to (P,\leq)^{\mathrm{op}}.$$
\end{df}

The entrance path category $\mathcal{E}(P)$ encodes not just the relations in the poset $P$, but also the different ways in which relations can be composed, capturing richer `homotopical' information. This motivation should seem very natural to readers familiar with the reasons for studying $\infty$-categories, often regarded as `homotopy coherent category theory'.

Our definition of acyclic partial matchings can now be equivalently rephrased as a selection of atomic morphisms in $\mathcal{E}(K,\leq)$, whose sources and targets are pairwise disjoint, subject to an analogue of the acyclicity condition.

This viewpoint of partial matchings inspires Nanda to define the notion of a \emph{Morse system} on a more general cellular category. These are defined in a novel way using two extra axioms (the \emph{lifting} and \emph{switching} axioms) but do indeed subsume the notion of a partial matching that we defined earlier. Furthermore, given a Morse system $\Sigma$ on a cellular category, we can define the notion a $\Sigma$-\emph{critical object} which agrees with our definition of $\Sigma$-critical simplices.

\begin{df}
  Given a collection (even a proper class) of morphisms $\mathcal{W}$ in a category $\mathcal{C}$, we can \emph{localise} $\mathcal{C}$ at the morphisms $\mathcal{W}$, denoted $\mathcal{C}[\mathcal{W}^{-1}]$, which we think of as the category $\mathcal{C}$ with all its morphisms formally inverted. This is described rigorously as a functor $$L_{\mathcal{W}}:\mathcal{C}\overset{}{\longrightarrow}\mathcal{C}[\mathcal{W}^{-1}]$$satisfying the \emph{universal property of the localisation}, which states that any functor $F:\mathcal{C}\to \mathcal{D}$ that sends all morphisms $\mathcal{W}$ to isomorphisms in $\mathcal{D}$, admits a unique factorisation through $L$:
  \[
    \begin{tikzcd}[scale=1, column sep=5mm, row sep=7mm]
      \mathcal{C} \ar[rr, "F"] \ar[dr, "L"']  && \mathcal{D} \\
      & \mathcal{C}[\mathcal{W}^{-1}] \ar[ur, "\exists!"', dashed]
    \end{tikzcd}
  \]
\end{df}

\begin{df}

  Given a cellular category $E$ equipped with a Morse system $\Sigma$, we define the \emph{discrete flow category} $\mathsf{Flo}_{\Sigma}E$ as the full subcategory of $\mathsf{Loc}_{\Sigma}E$ generated by the $\Sigma$-critical objects.
\end{df}

Nanda goes on to remark that the inclusion map $$J_{\Sigma}:\mathsf{Flo}_{\Sigma}E\hookrightarrow\mathsf{Loc}_{\Sigma}E$$ does not necessarily induce a homotopy equivalence on the \emph{classifying spaces} of these poset enriched categories, where the classifying space is obtained via the \emph{geometric nerve} functor (followed by the geometric realisation functor of a simplicial set), denoted $\Delta$. This nerve functor was first detailed in \cite{Dus01}, and was subsequently shown (in \cite{BC03}) to produce classifying spaces homotopy equivalent to those induced by other known nerve functors on poset enriched categories.

To remedy this, Nanda imposes additional hypotheses on the definition of a Morse system that ensure the inclusion map does actually induce a homotopy equivalence of classifying spaces. Nanda calls these \emph{mild Morse systems}, and in section 4, shows that Morse systems arising from acyclic partial matchings on finite CW complexes (hence simplicial complexes) are in fact always \emph{mild}. With these definitions introduced, Nanda's paper uses an argument rooted in Quillen's theorem A (from \cite{Qui73}, now ubiquitous in most areas of topology) to obtain the following result.

\begin{thm}\label{theorem:5.4}
  Let $E$ be a cellular category equipped with a Morse system $\Sigma$, and let $\mathsf{Flo}_{\Sigma}E$ be the the full subcategory of the localisation $\mathsf{Loc}_{\Sigma}E$ generated by the $\Sigma$-critical objects. Then, the localisation functor on poset enriched categories, $$L_{\Sigma}:E\overset{}{\longrightarrow} \mathsf{Loc}_{\Sigma}E,$$induces a homotopy-equivalence on their corresponding classifying spaces.

  Furthermore, if $\Sigma$ is mild, the inclusion functor, $$J:\mathsf{Flo}_{\Sigma}E\overset{}{\longrightarrow} \mathsf{Loc}_{\Sigma}E,$$also induces a homotopy equivalence of classifying spaces.

\end{thm}

The last section of \cite{Nan19} mentions how the above theorem can be applied to cellular cosheaf homology (for which simplicial cosheaf homology is a special case of). We now outline this application for the case of simplicial complexes. First we state another result, found in section 3 of \cite{Nan19}.

\begin{prop}\label{proposition:5.5}
  For any finite regular CW complex $X$, there is a homotopy equivalence between $X$ and the classifying space of its entrance path category, $$X\overset{\sim }{\longrightarrow}{|\Delta \mathcal{E}(X)|}.$$
\end{prop}

Now, given a cosheaf $\mathscr{C}$ on $K$ with a compatible acyclic partial matching, we consider the functor $$\mathscr{C}\circ \Pi:\mathcal{E}(K,\leq)\to \mathsf{Vect}^{\mathrm{fd}}_{\mathbb{F}},$$where $\Pi$ is the projection functor described in \Cref{definition:5.1}. Then the mild Morse system $\Sigma$ (induced by the matching) allows us to uniquely factor the aforementioned functor, as shown in the following diagram.

\[
  \begin{tikzcd}[scale=1, column sep=8mm, row sep=7mm]
    & \mathcal{E}(K,\leq) \ar[dr, "\mathscr{C}\Pi"] \ar[d, "L_{\Sigma}"'] \\
    \mathsf{Flo}_{\Sigma}\mathcal{E}(K,\leq) \ar[r, "J_{\Sigma}"] & \mathsf{Loc}_{\Sigma}\mathcal{E}(K,\leq) \ar[r, "\exists!\mathscr{C}'",dashed]  &
    \mathsf{Vect}^{\mathrm{fd}}_{\mathbb{F}}\;\;\;\;\;\;\;\;\;\;
  \end{tikzcd}
\]

To conclude, we need some intuition for the objects involved. We first note that the set of morphisms between two simplices in $\mathcal{E}(K,\leq)$ combinatorially encodes all the downward pointing paths in $K$ between them. This passes the \emph{global} path data of $(K,\leq)$ to \emph{local} data between any two objects. We can then think of the localisation at $\Sigma$ as adding a formal inverse to the incident simplices that we can safely remove with a sequence of elementary collapses. The reason we needed to use $\mathcal{E}(K,\leq)$, is that it allows us (after passing to the localisation and restricting to the critical simplices) to retain the important information of zig-zags that move through incident matched simplices. These zig-zags are what we used to calculate the differentials in the Morse complex. When we take the discrete flow subcategory of the localised category, we remove all the objects corresponding to matched simplices, allowing us to reduce the dimensions of the vector spaces in the chain complex. Crucially, we also retain the information necessary to consider zig-zag paths, which remain encoded in the Hom sets of the critical simplices.\footnote{A more detailed description of Nanda's discrete flow category can be found in \cite{Hem24}, which also produces and implements an algorithm for computing the flow category's Hom posets.}

This intuition is the starting point towards fleshing out, in a rigorous manner, that the cosheaf $\mathscr{C}'J_{\Sigma}$ on $\mathsf{Flo}_{\Sigma}\mathcal{E}(K,\leq)$ produces a chain complex\footnote{This construction involves homological algebra which would not fit well in to this essay. The details can be found in \cite{Hem24}.} which agrees with the Morse complex `$\mathscr{M_{\bullet}}$' associated to $C_{\bullet}(K;\mathscr{C})$: $$H_{\bullet}(\mathsf{Flo}_{\Sigma}\mathcal{E}(K,\leq);\mathscr{C}'J_{\Sigma})\cong H_{\bullet}(\mathscr{M_{\bullet}}).$$So the homotopy invariances of \Cref{theorem:5.4} and \Cref{proposition:5.5} allow us to obtain the sequence of isomorphisms $$H_{\bullet}(\mathscr{M_{\bullet}})\cong H_{\bullet}(\mathsf{Flo}_{\Sigma}\mathcal{E}(K,\leq);\mathscr{C}'J_{\Sigma})\cong H_{\bullet}(\mathcal{E}(K\leq),\mathscr{C}P)\cong H_{\bullet}(K;\mathscr{C})$$Hence showing the cosheaf homology of $\mathscr{C}$ is recovered by the Morse complex as required.

To properly explicate the argument presented above would require delving into theory beyond the scope of this essay, such as spectral sequences, and more homotopy theory. The author believes (though is still in the process of reading) \cite{GJ99} could be a good starting point for the homotopy theory and \cite{Seg68} seems to be relevant for the spectral sequences (however the author has not yet read this).

\begin{rem}
  The author notes that several concepts found within this perspective of discrete Morse theory (such as entrance path categories and localisations) hint at underlying higher categorical structures. This motivates the author to try reformulating these ideas within the framework of $\infty$-categories, perhaps using models like the simplicial nerve to capture homotopy-coherent information more fully than the 2-categorical approach using entrance paths. Potentially, a further direction could be to investigate whether the DMT's computational purpose to `compress' chain complex data can be generalised to simplify computations within abstract cohomology theories pertinent to higher category theory (such as those described in chapter seven of \cite{Lur09}). While abstract higher category theory might initially seem computationally challenging, the most commonly used model (quasicategories - simplicial sets satisfying the weak Kan condition) is fundamentally combinatorial. Although these simplicial objects are generally infinite, they often admit finite presentations (e.g., nerves of finitely presented categories) or possess sufficient structure to allow computation, suggesting the potential feasibility of developing such a generalised Morse-theoretic compression mechanism. Unfortunately, the author will need to read a lot more to understand if this is worth exploring.
\end{rem}

\section{The Morse Mayer-Vietoris sequence}\label{section:6}

In this section, we will see how the Morse complex can be united with the Mayer-Vietoris sequence, forming the \emph{Morse Mayer-Vietoris} short exact sequence. To do so, we must first impose an extra compatibility constraint on the type of partial matchings we consider.

\begin{df}
  Let $M$ be a subcomplex of $K$, and let $\Sigma$ be a $\mathscr{C}$-compatible acyclic partial matching on $K$. We say $\Sigma$ is \emph{compatible with} $M$ if for every matched pair $(\sigma\lhd\tau) \in \Sigma$, we have $\sigma \in M$ if and only if $\tau \in M$. We define the restriction of $\Sigma$ to $M$ by $$\Sigma_{M}:=\{(\sigma\lhd\tau)\in \Sigma \mid \tau,\sigma \in M \}.$$
\end{df}

Note that as $M$ is a subcomplex, it is downward closed with respect to the face ordering, so the `only if' part of the above definition is the only new constraint. More intuitively, we are requiring that the partial matching flows \emph{into} $M$, never out of $M$. Now we form the Morse Mayer-Vietoris sequence.

\begin{thm}
  Suppose we have a cosheaf $\mathscr{C}$ on $K$, a subcomplex decomposition $K=L\cup M$ \emph{(}writing $I=L\cap M$\emph{)}, and an acyclic partial matching $\Sigma$ which is compatible with both $\mathscr{C}$ and the subcomplexes $L$ and $M$. Then we can form the \emph{Morse Mayer-Vietoris} short exact sequence: $$0\overset{}{\longrightarrow}\mathscr{M}_{\bullet}^{I}\overset{}{\longrightarrow}\mathscr{M}_{\bullet}^{L}\oplus\mathscr{M}_{\bullet}^{K}\overset{}{\longrightarrow}\mathscr{M}_{\bullet}^{K}\overset{}{\longrightarrow}0.$$Here $\mathscr{M}_{\bullet}^{N}$ denotes the Morse complex associated to $C_{\bullet}(N;\mathscr{C}|_{N})$ and the sequence is quasi-isomorphic to the Mayer-Vietoris sequence of \Cref{theorem:3.2}.
\end{thm}

\begin{proof}
  We note that we don't have to impose the compatibility constraint on $I$ as this is implied by the constraint on $L$ and $M$. The compatibility constraint on any pair of subcomplexes, say $B\subseteq A$, ensures that critical simplices of $B$ remain critical in $A$, allowing us to obtain an inclusion map $$\iota_{\bullet}^{B,A}:\mathscr{M}_{\bullet}^{B}\hookrightarrow\mathscr{M}_{\bullet}^{A}.$$
  The compatibility constraint ensures that there are no $\Sigma_{A}$-paths starting at a face of a $B$-simplex and ending outside of $B$. This implies that $[\alpha:\omega]_{\Sigma}$ is zero for all $\alpha \in B$ and $\omega \in A\backslash B$. Hence, the inclusion forms a chain map:
  \[
    \begin{tikzcd}[scale=1, column sep=12mm, row sep=9mm]
      \cdots \ar[r, "\partial_{k+1}^{\Sigma_{B}}"] & \mathscr{M}^{B}_{k} \ar[d, "\iota^{B,A}_{k}",hookrightarrow]  \ar[r, "\partial_{k}^{\Sigma_{B}}",hookrightarrow] & \mathscr{M}^{B}_{k-1} \ar[d, "\iota^{B,A}_{k-1}",hookrightarrow] \ar[r, "\partial_{k-1}^{\Sigma_{B}}"] & \cdots \\
      \cdots \ar[r, "\partial_{k+1}^{\Sigma_{A}}"]  & \mathscr{M}^{A}_{k}  \ar[r, "\partial_{k}^{\Sigma_{A}}"] & \mathscr{M}^{A}_{k-1} \ar[r, "\partial_{k-1}^{\Sigma_{A}}"] & \cdots
    \end{tikzcd}
  \]
  Furthermore, the inclusion maps agree with the standard inclusions, $$C_{\bullet}(B;\mathscr{C}|_{B})\hookrightarrow C_{\bullet}(A;\mathscr{C}|_{A}),$$ used in the proof of \Cref{theorem:3.2}. This observation gives us that the inclusion maps commute with the chain homotopy equivalence $\mathscr{M}_{\bullet}\simeq C_{\bullet}^{\mathscr{C}}$:

  \[
    \begin{tikzcd}[scale=1, column sep=9mm, row sep=7mm]
      C_{\bullet}(B;\mathscr{C}|_{B}) \ar[r, "i_{\bullet}^{B,A}",hookrightarrow] \ar[d, "\sim" vertical,leftrightarrow] & C_{\bullet}(A;\mathscr{C}|_{A}) \ar[d, "\sim" vertical,leftrightarrow]  \\
      \mathscr{M}_{\bullet}^{B} \ar[r, "\iota_{\bullet^{B,A}}",hookrightarrow] & \mathscr{M}_{\bullet}^{A}
    \end{tikzcd}
  \]

  Hence, applying this to the subcomplexes $I,L,M$ of $K$, yields the commutative cube diagram:

  \[
    \begin{tikzcd}[row sep={50,between origins}, column sep={50,between origins}]
      & C_{\bullet}(I;\mathscr{C}|_{I})
      \ar[rr, "i_{\bullet}^{I,M}", hookrightarrow]
      \ar[dl, "i_{\bullet}^{I,L}"', hookrightarrow, outer sep=-3]
      \ar[dd, "\sim" vertical, leftrightarrow]
      & & C_{\bullet}(M;\mathscr{C}|_{M})
      \ar[dl, "i_{\bullet}^{M,K}"', hookrightarrow, outer sep=-4]
      \ar[dd, "\sim" vertical, leftrightarrow]
      \\
      C_{\bullet}(L;\mathscr{C}|_{L})
      \ar[rr, crossing over, phantom, shift left=2]
      \ar[rr, crossing over, phantom, shift left=4]
      \ar[rr, crossing over, "i_{\bullet}^{L,K}",hookrightarrow,outer sep=2]
      \ar[dd, "\sim" vertical, leftrightarrow]
      & & C_{\bullet}(K;\mathscr{C}|_{K})
      \\
      & \mathscr{M}_{\bullet}^{I}
      \ar[rr, "\iota_{\bullet}^{I,M}" pos=0.8, hookrightarrow]
      \ar[dl, "\iota_{\bullet}^{I,L}"' pos=0.3, hookrightarrow, outer sep=-2]
      & & \mathscr{M}_{\bullet}^{M}
      \ar[dl, "\iota_{\bullet}^{M,K}", hookrightarrow, outer sep=-2]
      \\
      \mathscr{M}_{\bullet}^{L}
      \ar[rr, crossing over, "\iota_{\bullet}^{L,K}"', hookrightarrow]
      & & \mathscr{M}_{\bullet}^{K}
      \ar[uu, crossing over, phantom, shift left=2]
      \ar[uu, crossing over, "\sim" vertical, leftrightarrow]
    \end{tikzcd}
  \]

  Finally, we take the direct sum across the $L$ and $M$ vertices of the top and bottom faces of the cube and negate the $M\subseteq K$ inclusion maps to yield the sequence:

  \[
    \begin{tikzcd}[scale=1, column sep=10mm, row sep=7mm]
      0 \ar[r, ""] &
      C_{\bullet}(I;\mathscr{C}|_{I})
      \ar[r, "i^{I,L}_{\bullet}\oplus i^{I,M}_{\bullet}", outer sep=3, hookrightarrow, pos=0.7] &
      C_{\bullet}(L;\mathscr{C}|_{L})\oplus C_{\bullet}(M;\mathscr{C}|_{M})
      \ar[r, "i^{L,K}_{\bullet}\oplus -i^{M,K}_{\bullet}", outer sep=3,twoheadrightarrow, pos=0.7] &
      C_{\bullet}(K;\mathscr{C}|_{K}) \ar[r, ""] & 0 \\
      0 \ar[r, ""] &
      \mathscr{M}_{\bullet}^{I}
      \ar[r, "\iota^{I,L}_{\bullet}\oplus\iota^{I,M}_{\bullet}", hookrightarrow] \ar[u, "\sim" vertical, leftrightarrow]  &
      \mathscr{M}_{\bullet}^{L}\oplus\mathscr{M}_{\bullet}^{M}
      \ar[r, "\iota^{L,K}_{\bullet}\oplus{-\iota^{M,K}_{\bullet}}", twoheadrightarrow] \ar[u, "\sim" vertical, leftrightarrow] &
      \mathscr{M}_{\bullet}^{K} \ar[r, ""] \ar[u, "\sim" vertical, leftrightarrow] & 0
    \end{tikzcd}
  \]

  The top row is the Mayer-Vietoris sequence of \Cref{theorem:3.2} and so we have formed the Morse Mayer-Vietoris sequence, and shown it to be quasi-isomorphic to the Mayer-Vietoris sequence (and hence isomorphic under the homology functor) as required.
\end{proof}

\begin{cor}
  We have isomorphic long exact sequences on Homology:
  \[
    \begin{tikzcd}[scale=0.8, column sep=5mm]
      \cdots \ar[r] &
      H_{k}(I;{\mathscr{C}}|_{I}) \ar[r] \ar[d, "\cong" vertical, leftrightarrow] &
      H_{k}(L;{\mathscr{C}}|_{L})\oplus H_{k}(M;{\mathscr{C}}|_{M}) \ar[r] \ar[d, "\cong" vertical, leftrightarrow] &
      H_{k}(K;\mathscr{C}) \ar[r] \ar[d, "\cong" vertical, leftrightarrow] &
      H_{k-1}(I;{\mathscr{C}}|_{I}) \ar[r] \ar[d, "\cong" vertical, leftrightarrow] & \cdots \\
      \cdots \ar[r] &
      H_{k}(\mathscr{M}^{I}_{\bullet}) \ar[r] &
      H_{k}(\mathscr{M}^{L}_{\bullet})\oplus H_{k}(\mathscr{M}^{M}_{\bullet}) \ar[r] &
      H_{k}(\mathscr{M}^{K}_{\bullet}) \ar[r] &
      H_{k-1}(\mathscr{M}^{I}_{\bullet}) \ar[r] & \cdots \\
    \end{tikzcd}
  \]
\end{cor}

\section{Conclusion}

In conclusion, we have demonstrated the construction of a Mayer-Vietoris sequence for cosheaf homology and explored its simplification using Discrete Morse Theory. The application of these techniques could potentially yield significant computational advantages. Standard algorithms for computing homology typically exhibit super-linear, often polynomial, dependence on the number of simplices in the complex. The standard Mayer-Vietoris sequence provides us with a strategy for reducing the computation on a large complex $K=L\cup M$ to computations on smaller complexes $L$, $M$, and $I=L\cap M$, thereby potentially lowering the computational cost associated with the input size. Unfortunately, the computation of the connecting morphisms in the associated long exact sequence may still be very computationally complex.

On the other hand, (as discussed in an earlier \hyperref[rmk:scythe]{remark}) Discrete Morse Theory provides a significant computational advantage by constructing a quasi-isomorphic chain complex $\mathscr{M}_{\bullet}$, which is often dramatically smaller than the original chain complex $C_{\bullet}(K;\mathscr{C})$, as mentioned.

The Morse Mayer-Vietoris sequence has the potential to unify both these advantages. It provides a short exact sequence involving the smaller Morse complexes, which is quasi-isomorphic to the standard sequence. This allows the associated long exact sequence in homology to be computed using these compressed complexes. Assuming such a sequence has been constructed, it is likely to afford a simpler computation of the connecting homomorphisms in the long exact sequence. While the formula for the Morse boundary map $\partial^{\mathscr{C},\Sigma}_{\bullet}$ involves sums over potentially complex gradient paths, \cite{CGN16} has already demonstrated an efficient algorithm for computing these maps. Computing the connecting morphisms in the associated long exact sequence is likely to be much more efficient than in the standard Mayer-Vietoris sequence due to the reductions in dimension.

Future work could involve developing algorithms to efficiently find acyclic partial matchings that are compatible with given decompositions, or indeed algorithms to find suitable decompositions themselves. Further exploration of the categorical and homotopy-theoretic perspectives on discrete Morse theory, as introduced in \Cref{section:5}, particularly concerning potential $\infty$-categorical formulations and generalisations to other cohomology theories, may also be an exciting area of investigation.

\bibliography{bib.bib}
\bibliographystyle{amsalpha}
\end{document}